\documentclass[a4paper,10pt, reqno]{amsart}
\usepackage{bbm}

\numberwithin{equation}{section}

\usepackage{mathrsfs}
\usepackage{amsmath,amssymb, amsthm}
\usepackage[all,poly,knot]{xy}
\usepackage[colorlinks,linkcolor=black,anchorcolor=black,citecolor=black]{hyperref}
\usepackage{color}

\newtheorem{thm}{Theorem}[section]
\newtheorem{defi}[thm]{{Definition}}
\newtheorem{cor}[thm]{{Corollary}}
\newtheorem{lem}[thm]{{Lemma}}
\newtheorem{prop}[thm]{Proposition}

\newtheorem{exa}[thm]{{Example}}

\newtheorem{note}[thm]{{Notation}}

\def\C{\mathscr C}
\def\Ac{\mathrm{Aut}_{\C}}

\def\Hom{\mathrm{Hom}}
\def\Homc{\Hom_{\C}}
\def\Ho{\Hom_{\C}^0}

\def\Id{\mathrm{Id}}

\def\Mor{\mathrm{Mor}}
\def\Obj{\mathrm{Obj}}

\def\id{\mathrm{id}}
\def\pd{\mathrm{pd}}

\def\lbr{\left(\begin{array}{c}}
\def\lbrt{\left(\begin{array}{cc}}
\def\lbrth{\left(\begin{array}{ccc}}
\def\rbr{\end{array}\right)}

\title[Gorenstein triangular matrix rings and category algebras]{Gorenstein triangular matrix rings and category algebras}
\author{Ren Wang}
\address{School of Mathematical Sciences, University of Science and Technology of China, Hefei, Anhui 230026, P. R. China}
\email{renw@mail.ustc.edu.cn}
\date{\today}

\begin{document}
\begin{abstract}
We give conditions on when a triangular matrix ring is Gorenstein of
a given selfinjective dimension. We apply the result to the category
algebra of a finite EI category. In particular, we prove that for a
finite EI category, its category algebra is $1$-Gorenstein if and
only if the given category is free and projective.
\end{abstract}

\maketitle

\section{Introduction}

Let $R$ be a ring with a unit. Recall that $R$ is \emph{Gorenstein}
if $R$ is two-sided noetherian satisfying $\id_RR < \infty$ and $\id
R_R < \infty$. Here, we use $\id$ to denote the injective dimension
of a module. It is well known that for a Gorenstein ring $R$ we have
$\id_RR=\id R_R$; see \cite[Lemma A]{Zaks}. Let $m\geq 0$. A
Gorenstein ring $R$ is \emph{$m$-Gorenstein} if $\id_RR=\id R_R \leq
m$. We observe that a $0$-Gorenstein ring coincides with a
quasi-Frobenius ring.

 Let $n\geq 2$. Let $\Gamma=\left(
          \begin{array}{cccc}
            R_1 & M_{12} & \cdots & M_{1n} \\
             & R_2 & \cdots & M_{2n} \\
             &  & \ddots & \vdots \\
                &  &  & R_n \\
          \end{array}
        \right)$ be an upper triangular matrix ring of order $n$, where each $R_i$ is a ring and each $M_{ij}$ is an $R_i$-$R_j$-bimodule
        together with bimodule morphisms $\psi_{ilj}: M_{il}\otimes_{R_l} M_{lj} \rightarrow M_{ij}$ satisfying
\[\psi_{ijt}(\psi_{ilj}(m_{il}\otimes m_{lj})\otimes m_{jt})=\psi_{ilt}(m_{il}\otimes \psi_{ljt}(m_{lj}\otimes m_{jt}))\]
for $1\leq i<l<j<t\leq n$.

For $1\leq t\leq n-1$, we denote by $\Gamma_t=\left(
          \begin{array}{cccc}
            R_1 & M_{12} & \cdots & M_{1t} \\
             & R_2 & \cdots & M_{2t} \\
             &  & \ddots & \vdots \\
                &  &  & R_t \\
          \end{array}
        \right)$ the ring given by the $t\times t$ leading principal
        submatrix of $\Gamma$. Denote the natural left
        $\Gamma_t$-module $\lbr M_{1,t+1}\\
\vdots \\ M_{t,t+1}\\ \rbr$ by $M_t^*$.

We will prove the following result, where the first statement is
obtained by applying \cite[Theorem 3.3]{XChen} repeatedly. We
mention that the Gorensteinness of an upper triangular matrix ring
is studied in \cite{XZ} and \cite{EC}.

\begin{prop}\label{XChen1'}
Assume that $R_1, R_2, \cdots, R_n$ are quasi-Frobenius rings. Then
\begin{enumerate}
\item The upper triangular matrix ring $\Gamma$ is Gorenstein if and only if all the bimodules $M_{ij}$ are finitely
generated projective on both sides;

\item The upper triangular matrix ring $\Gamma$ is $1$-Gorenstein if and only if all the bimodules $M_{ij}$ are finitely
generated projective on both sides, and each left $\Gamma_t$-module
$M_t^*$ is projective for $1\leq t\leq n-1$.
\end{enumerate}
\end{prop}

For the proof, (1) is a special case of Proposition \ref{XChen1}.
Thanks to (1), the ``only if" part of (2) is a special case of
Proposition \ref{LG1}(1), and the ``if"  part is a special case of
Proposition \ref{LG1}(2). Here, we use the fact that a module over a
quasi-Frobenius ring is projective provided that it has finite
projective dimension.

Let $k$ be a field and let $\C$ be a finite EI category. Here, the
EI condition means that any endomorphism in $\C$ is an isomorphism.
For an object $x$, we denote by $\Ac(x)$ and $k\Ac(x)$ the group of
endomorphisms of $x$ and the group algebra, respectively. We observe
that for any two objects $x$ and $y$ in $\C$, $k{\rm Hom}_{\C}(x,y)$
is a $k{\rm Aut}_{\C}(y)$-$k{\rm Aut}_{\C}(x)$-bimodule. We say that
a finite EI category $\C$ is \emph{projective over $k$} if each
bimodule $k{\rm Hom}_{\C}(x,y)$ is projective on both sides.

We denote by $k\C$ the category algebra of $\C$. We mention that
category algebras play an important role in the representation
theory of finite groups; see \cite{PWebb1,PWebb2}. The following
result is an application of Proposition \ref{XChen1'}(1), where we
use the fact that the category algebra is isomorphic to a certain
upper triangular matrix algebra.

\begin{prop}\label{G}
 Let $k$ be a field and $\C$ be a finite EI category. Then the category algebra $k\C$ is Gorenstein if and only if $\C$ is projective over $k$.
\end{prop}

The concept of a finite \emph{free} EI category is introduced in
\cite{LLi}. It is due to \cite[Theorem 5.3]{LLi} that the category
algebra $k\C$ is hereditary if and only if $\C$ is a finite free EI
category satisfying that the endomorphism groups of all objects have
orders invertible in $k$. The following result is a Gorenstein
analogue to \cite[Theorem 5.3]{LLi}.

\begin{thm}\label{1G}
 Let $k$ be a field and $\C$ be a finite EI category. Then the category algebra $k\C$ is $1$-Gorenstein if and only if  $\C$ is free and projective over $k$.
\end{thm}

Indeed, we may deduce \cite[Theorem 5.3]{LLi} from Theorem \ref{1G},
using the well-known fact that a finite dimensional algebra is
hereditary if and only if it is $1$-Gorenstein with finite global
dimension; see Example \ref{EX}.

This paper is organized as follows. In Section 2,
   we give an explicit description of projective modules
  and injective modules over an upper triangular matrix ring.
  In Section 3, we give conditions on when a triangular matrix ring is Gorenstein of
a given selfinjective dimension, and prove
Proposition~\ref{XChen1'}. In Section 4, we give a new
characterization of finite free EI categories in terms of the
corresponding triangular matrix algebras; see Proposition
\ref{MatInterpOfFreeness}.
 In Section 5, we prove Proposition~\ref{G} and Theorem~\ref{1G}.

\section{Modules over triangular matrix rings}
In this section, we describe explicitly projective modules and
injective modules over an upper triangular matrix ring.

Let $R_1$ and $R_2$ be two rings, $M_{12}$ an $R_1$-$R_2$-bimodule.
We consider the corresponding upper triangular matrix ring
$\Gamma=\lbrt R_1& M_{12}\\ 0 & R_2 \rbr$.

Recall the description of left $\Gamma$-modules via column vectors.
Let $X_i$ be a left $R_i$-module, $i=1,2$, and let $\varphi_{12}:
M_{12}\otimes_{R_2} X_2 \rightarrow X_1$ be a morphism of left
$R_1$-modules. We define the left $\Gamma$-module structure on  $X=\lbr X_1\\
X_2\rbr $ by the following identity
 \[\lbrt r_1& m_{12}\\ 0 & r_2  \rbr \lbr x_1\\
x_2\rbr =\lbr r_1x_1+\varphi_{12}(m_{12}\otimes x_2) \\ r_2x_2
\rbr.\] We mention that the left $\Gamma$-module structure on  $X=\lbr X_1\\
X_2\rbr $ depends on the morphism $\varphi_{12}$. Indeed, every left
$\Gamma$-module arises in
 this way; compare
\cite[III.2, Proposition 2.1]{ARS}. A morphism $\lbr X_1\\
X_2\rbr \rightarrow \lbr X'_1\\
X'_2\rbr $ of $\Gamma$-modules is denoted by $\lbr f_1\\
f_2\rbr $, where $f_i:X_i \rightarrow X'_i$ is an $R_i$-morphism
satisfying
\begin{equation}\label{A}
f_1\circ \varphi_{12}=\varphi'_{12}\circ ({\rm Id}_{M_{12}}\otimes
f_2).
\end{equation}
Dually, we have the description of right $\Gamma$-modules via row
vectors.

Let $M$ be a left module over a ring $R$. We denote by $\pd_R M$ and
$\id_R M$ the projective dimension and the injective dimension of
$M$, respectively.

The following lemma is well-known; compare \cite[III, Proposition
2.3 and 2.5]{ARS} and \cite[Lemma 1.2]{XZ}.

\begin{lem}\label{ExtFun}
Let $\Gamma=\lbrt R_1& M_{12}\\ 0 & R_2  \rbr$ be an upper
triangular matrix ring, and let $X=\lbr X_1\\ X_2\rbr $  be a left
$\Gamma$-module as above. Then the following statements hold.
\begin{enumerate}
\item $\pd_{R_1} X_1=\pd_\Gamma \lbr X_1\\ 0\rbr$ and $\id_{R_2}
X_2=\id_\Gamma \lbr 0\\ X_2\rbr$.

\item $\pd_{R_2} X_2\leq \pd_\Gamma \lbr X_1\\ X_2\rbr$ and $\id_{R_1} X_1\leq \id_\Gamma \lbr X_1\\
 X_2\rbr$. \hfill $\square$
\end{enumerate}
\end{lem}

Let $\Gamma=\lbrt R_1& M_{12}\\ 0 & R_2  \rbr$ be as above. For each
left $R_1$-module $X_1$, we associate two left $\Gamma$-modules $i_1(X_1)=\lbr X_1\\
0\rbr $ and $j_1(X_1)=\lbr X_1\\ \Hom_{R_1}(M_{12},
 X_1) \rbr $, where the $\Gamma$-module structure on $j_1(X_1)$ is
determined by the evaluation map
$M_{12}\otimes_{R_2}\Hom_{R_1}(M_{12},
 X_1)\rightarrow X_1$. For each
left $R_2$-module $X_2$, we associate two left $\Gamma$-modules
$i_2(X_2)=\lbr M_{12}\otimes_{R_2}X_2\\ X_2\rbr
 $ and $j_2(X_2)=\lbr 0\\ X_2 \rbr $, where the $\Gamma$-module structure on $i_2(X_2)$ is
determined by the identity map on $M_{12}\otimes_{R_2}
 X_2$.

The following result seems to be well known; compare \cite[III,
Proposition 2.5]{ARS}. For completeness, we include a proof.

\begin{lem}\label{PIOrder2}
Let $\Gamma=\lbrt R_1& M_{12}\\ 0 & R_2  \rbr$ be an upper
triangular matrix ring. Then we have the following statements.
\begin{enumerate}
\item A left $\Gamma$-module is projective if and only if it is
isomorphic to $i_1(P_1)\oplus i_2(P_2)$ for some projective left
$R_1$-module $P_1$ and projective left $R_2$-module $P_2$.

\item A left $\Gamma$-module is injective if and only if it is
isomorphic to $j_1(I_1)\oplus j_2(I_2)$ for some injective left
$R_1$-module $I_1$ and injective left $R_2$-module $I_2$.
\end{enumerate}
\end{lem}

\begin{proof}
We only prove (1). For the ``if" part, we consider the
subring $\Gamma'=\lbrt R_1& 0\\ 0 & R_2  \rbr$ of $\Gamma$. Then $P'=\lbr P_1\\
P_2\rbr $ is a projective left $\Gamma'$-module. We observe that
$i_1(P_1)\oplus i_2(P_2)$ isomorphic to $\Gamma\otimes_{\Gamma'}P'$.
It follows that the left $\Gamma$-module $i_1(P_1)\oplus i_2(P_2)$
is projective.

For the ``only if" part, let $X=\lbr X_1\\
X_2\rbr $ be a projective left $\Gamma$-module. Then by
Lemma~\ref{ExtFun}(2) $X_2$ is a projective left $R_2$-module, and
by above $i_2(X_2)$ is a projective left
$\Gamma$-module. Consider the nature projections $i_2(X_2)\overset{\pi_1}{\twoheadrightarrow} \lbr 0 \\
X_2 \rbr$ and $X \overset{\pi_2}{\twoheadrightarrow} \lbr 0 \\
X_2 \rbr$. Since $i_2(X_2)$ and $X$ are projective left
$\Gamma$-modules, we have two morphisms
$i_2(X_2)\overset{\alpha}{\rightarrow} X$ and $X
\overset{\beta}{\rightarrow} i_2(X_2)$ satisfying
$\pi_1=\pi_2\circ\alpha$ and $\pi_2=\pi_1\circ\beta$. Therefore,
$\pi_1\circ \beta\circ\alpha=\pi_1$. By (\ref{A}), we observe that
$\beta\circ\alpha= {\rm Id}_{i_2(X_2)}$. It follows that $\alpha$ is
a split monomorphism. In particular, $X$ is isomorphic to
$i_2(X_2)\oplus {\rm Coker}\alpha$. We observe that ${\rm
Coker}\alpha$ is of the form $\lbr X'_1\\ 0\\ \rbr=i_1(X'_1)$ for
some left $R_1$-module $X'_1$. Since ${\rm Coker}\alpha$ is a
projective left $\Gamma$-module, we have that $X'_1$ is a projective
left $R_1$-module by Lemma~\ref{ExtFun}(1). Then we are done.
\end{proof}

 We now extend the above results to an upper triangular matrix ring of an arbitrary order.

Let $n\geq 2$. Let $R_i$ be a ring for $1\leq i\leq n$, and let
$M_{ij}$ be an $R_i$-$R_j$-bimodule for $1\leq i< j\leq n$. Let
$\psi_{ilj}: M_{il}\otimes_{R_l} M_{lj} \rightarrow M_{ij}$ be
morphisms of $R_i$-$R_j$-bimodules satisfying
\[\psi_{ijt}(\psi_{ilj}(m_{il}\otimes m_{lj})\otimes m_{jt})=\psi_{ilt}(m_{il}\otimes \psi_{ljt}(m_{lj}\otimes m_{jt}))\]
for $1\leq i<l<j<t\leq n$.
Then we have the corresponding $n\times n$ upper
triangular matrix ring $\Gamma=\left(
          \begin{array}{cccc}
            R_1 & M_{12} & \cdots & M_{1n} \\
             & R_2 & \cdots & M_{2n} \\
             &  & \ddots & \vdots \\
                &  &  & R_n \\
          \end{array}
        \right)$.
The elements of $\Gamma$, denoted by $(m_{ij})$, are $n\times n$
upper triangular matrices, where $m_{ij}\in M_{ij}$ for $i<j$ and
$m_{ii}\in R_i$. The multiplication is induced by those morphisms
$\psi_{ilj}$. We write $\psi_{ilj}(m_{il}\otimes m_{lj})$ as
$m_{il}m_{lj}$.

We describe left $\Gamma$-modules via column vectors. Let $X_i$ be a
left $R_i$-module for $1\leq i\leq n$, and $\varphi_{jl}:
M_{jl}\otimes_{R_l} X_l \rightarrow X_j$ be a morphism of left
$R_j$-modules satisfying
\[\varphi_{ij}\circ({\rm Id}_{M_{ij}}\otimes \varphi_{jl})=\varphi_{il}\circ(\psi_{ijl}\otimes {\rm Id}_{X_l})\]
for $1\leq i<j<l\leq n$.
Set $X=\lbr X_1 \\
 \vdots \\ X_n \\ \rbr $. Denote the elements of
  $X$ by $(x_i)$, where $x_i\in X_i$ for $1\leq i\leq n$. We write
  $\varphi_{ij}(m_{ij}\otimes x_j)$ as $m_{ij}x_j$. Then we define the left $\Gamma$-module structure on $X$ by the following
identity
 \[(m_{ij}) (x_i)=(\sum_{l=i}^n m_{il}x_l).\]
Indeed, every left $\Gamma$-module arises in
 this way. 
 Dually, we have the description
of right $\Gamma$-modules via row vectors.

Let $\Gamma$ be an upper triangular matrix ring of order $n$ as
above. We consider the subring $\Gamma^D$= diag($R_1,\cdots R_n$) of
$\Gamma$ consisting of diagonal matrices. Observe that $\Gamma^D$ is
isomorphic to the direct product ring $\prod_{i=1}^n R_i$. We view
an $R_i$-module as a $\Gamma^D$-module via the projection
$\Gamma^D\twoheadrightarrow R_i$. We denote the $i$-th column of
$\Gamma$ by $C_i$ which is a $\Gamma$-$R_i$-bimodule, and denote the
$i$-th row of $\Gamma$ by $H_i$ which is an $R_i$-$\Gamma$-bimodule.

Let $1\leq t\leq n$, and let $A$ be a left $R_t$-module. We consider
the $\Gamma$-module $i_t(A)=C_t \otimes_{R_t} A$.
 Observe an isomorphism
\begin{align}\label{ISO}
\Gamma \otimes_{\Gamma^D} A \overset{\sim}{\longrightarrow} i_t(A)
\end{align} of $\Gamma$-modules sending $(m_{ij})\otimes a$ to
$(m_{it})\otimes a$. We describe $i_t(A)$ as $\left(
    \begin{array}{c}
      M_{1t} \otimes_{R_t} A \\
      \vdots \\
      R_t \otimes_{R_t} A \\
      \vdots \\
      0 \\
    \end{array}
  \right)$, where the corresponding morphisms $\varphi_{jl}=\psi_{jlt}\otimes {\rm Id}_{A}$ for
  $l<t$, $\varphi_{jt}: M_{jt}\otimes_{R_t}(R_t\otimes_{R_t}A)\rightarrow
  M_{jt}\otimes_{R_t}A$ is the canonical isomorphism, and $\varphi_{jl}=0$ for $l>t$.

We consider the $\Gamma$-module $j_t(A)=\Hom_{R_t}(H_t, A)$. Observe
an isomorphism
\begin{align}\label{ISO1}
\Hom_{\Gamma^D}(\Gamma, A)\overset{\sim}{\longrightarrow} j_t(A)
\end{align} of
$\Gamma$-modules sending $f$ to its restriction on $H_t$. We
describe $j_t(A)$ as $\left(
    \begin{array}{c}
     0\\
     \vdots \\
     {\Hom_{R_t}}(R_t, A)\\
      \vdots \\
     {\Hom_{R_t}}(M_{tn}, A)\\
    \end{array}
  \right)$, where the corresponding morphisms $\varphi_{jl}: M_{jl}\otimes_{R_l} {\Hom_{R_t}}(M_{tl},
A)\rightarrow {\Hom_{R_t}}(M_{tj}, A)$ are given by
$\varphi_{jl}(m_{jl}\otimes f)(m_{tj})=f(m_{tj}m_{jl})$ for $j>t$,
$\varphi_{tl}: M_{tl}\otimes_{R_l} {\Hom_{R_t}}(M_{tl},
A)\rightarrow {\Hom_{R_t}}(R_{t}, A)$ is the evaluation map, and
$\varphi_{jl}=0$ for $j<t$.


The following results give an explicit description of projective
modules and injective modules over an upper triangular matrix ring.

\begin{prop}\label{ProjInjOverTria}
Let $\Gamma$ be an upper triangular matrix ring of order $n$. Then
we have the following statements.
\begin{enumerate}
\item A left $\Gamma$-module is projective if and only if it is
isomorphic to $\bigoplus\limits_{t=1}^n i_t(P_t)$ for some
projective left $R_t$-module $P_t$, $1\leq t\leq n$.

\item A left $\Gamma$-module is injective if and only if it is
isomorphic to $\bigoplus\limits_{t=1}^n j_t(I_t)$ for some injective
left $R_t$-module $I_t$, $1\leq t\leq n$.
\end{enumerate}
\end{prop}

\begin{proof}
We only prove (1). The ``if" part is obvious since $i_t$ preserves
projective modules by the isomorphism (\ref{ISO}).

For the ``only if" part, we use induction on $n$. If $n=2$, it is
Lemma~\ref{PIOrder2}(1). Assume that $n> 2$. Write $\Gamma=\lbrt
\Gamma_{n-1} & M^*_{n-1}\\ 0 & R_n \rbr$, where $\Gamma_{n-1}$ is
the $(n-1)\times
(n-1)$ leading principal submatrix of $\Gamma$ and $M^*_{n-1}=\lbr M_{1n} \\
 \vdots \\ M_{n-1n} \rbr$ is a $\Gamma_{n-1}$-$R_n$-bimodule.
Assume that $X$ is a projective left
$\Gamma$-module. Write $X=\lbr X' \\ X_n \rbr$, where $X'=\lbr X_1 \\ \vdots \\
X_{n-1} \rbr$ is a left $\Gamma_{n-1}$-module. By
Lemma~\ref{PIOrder2}(1), $X\simeq i'_1(X'_1)\oplus i_n(P_n)$, where
$X'_1$ is a projective left $\Gamma_{n-1}$-module and $P_n$ is a
projective left $R_n$-module. By induction, we have an isomorphism
$X'_1\simeq \bigoplus\limits_{t=1}^{n-1} i_t(P_t)$ of
$\Gamma_{n-1}$-modules, where $P_t$ is a projective left
$R_t$-module, $1\leq t\leq n-1$. We identify $i'_1i_t(P_t)$ with
$i_t(P_t)$. This completes the proof.
\end{proof}

\section{Gorenstein triangular matrix rings}
In this section, we study Gorenstein upper triangular matrix rings.
We give conditions such that the upper triangular matrix rings are
Gorenstein with a prescribed selfinjective dimension.

Let $R$ be a ring with a unit. Recall that $R$ is \emph{Gorenstein}
if $R$ is two-sided noetherian satisfying $\id_RR < \infty$ and $\id
R_R < \infty$. It is well known that for a Gorenstein ring $R$,
$\id_RR=\id R_R$; see \cite[Lemma A]{Zaks}. Let $m\geq 0$. A
Gorenstein ring $R$ is \emph{$m$-Gorenstein} if $\id_RR=\id R_R \leq
m$. Recall that for any $m$-Gorenstein ring $R$ and any left (or
right) $R$-module $X$, $\id_RX < \infty$ if and only if $\id_RX \leq
m$, if and only if $\pd_RX < \infty$, if and only if $\pd_RX \leq
m$; see \cite[Theorem 9]{YIw}.

The following result generalizes \cite[Lemma 2.6]{XZ} with a
different proof.

\begin{lem}\label{LE}
Let $\Gamma=\left(
              \begin{array}{cc}
                R_1 & M_{12} \\
                  0 &    R_2 \\
              \end{array}
            \right)$  be a Gorenstein upper triangular matrix ring. Then the following are
            equivalent.
\begin{enumerate}
\item  $\id (R_1){_{R_1}} < \infty$.
\item The ring $R_1$ is Gorenstein.
\item $\pd {_{R_1}}M_{12} < \infty$.
\item $\id {_{R_2}}R_2 < \infty$.
\item The ring $R_2$ is Gorenstein.
\item $\pd (M_{12})_{R_2}<\infty$.

\end{enumerate}
\end{lem}

\begin{proof}
We observe that $R_1$ and $R_2$ are two-sided noetherian rings,
since they are isomorphic to certain quotient rings of $\Gamma$.

``(1)$\Rightarrow$ (2)" We observe that $\id_\Gamma \lbr R_1 \\
0 \rbr <\infty$, since $\lbr R_1
\\  0 \rbr$ is a projective left $\Gamma$-module and $\Gamma$ is Gorenstein. Lemma~\ref{ExtFun}(2) implies that $\id {_{R_1}}R_1 <\infty$. Hence
$R_1$ is Gorenstein.

``(2)$\Rightarrow$ (3)" We observe that $\id_\Gamma \lbr M_{12}\\ R_2\\
\rbr < \infty$, since $\lbr M_{12}\\
R_2\\ \rbr$ is a projective left $\Gamma$-module and $\Gamma$ is
Gorenstein. Lemma~\ref{ExtFun}(2) implies that $\id {_{R_1}}M_{12} <
\infty$. Since $R_1$ is Gorenstein, we have $\pd {_{R_1}}M_{12} <
\infty$ .

``(3)$\Rightarrow$ (4)" By Lemma~\ref{ExtFun}(1), $\pd_\Gamma\lbr
M_{12}\\0\rbr = \pd{_{R_1}}M_{12}< \infty$. Since $\Gamma$ is
Gorenstein, we have $\id_\Gamma\lbr M_{12}\\0\rbr< \infty$. Recall
from above that $\id_\Gamma \lbr M_{12}\\ R_2\\
\rbr < \infty$.
 The following exact sequence of $\Gamma$-modules
\[0\rightarrow \lbr M_{12} \\ 0 \\ \rbr
   \rightarrow \lbr M_{12} \\ R_2 \\ \rbr
   \rightarrow \lbr 0 \\ R_2 \\ \rbr \rightarrow 0\]
   implies $\id_\Gamma \lbr 0\\ R_2\rbr <\infty$. By Lemma~\ref{ExtFun}(1)
   $\id{_{R_2}}R_2<\infty$. Then we are done.

We observe that the opposite ring $\Gamma^{\rm op}$ of $\Gamma$ is
identified with the upper triangular matrix ring $\left(
              \begin{array}{cc}
                R_2^{\rm op} & M_{12} \\
                  0 &    R_1^{\rm op} \\
              \end{array}
            \right).$ Then
``(4)$\Rightarrow$ (5)" is similar to ``(1)$\Rightarrow$ (2)";
``(5)$\Rightarrow$ (6)" is similar to ``(2)$\Rightarrow$ (3)";
``(6)$\Rightarrow$ (1)" is similar to ``(3)$\Rightarrow$ (4)".
\end{proof}

\begin{lem}{\rm(\cite[ Thorem 3.3]{XChen})}\label{XChen}
Let $\Gamma=\left(
              \begin{array}{cc}
                R_1 & M_{12} \\
                 0  &    R_2 \\
              \end{array}
            \right)$ be an upper triangular matrix ring.
Assume that $R_1$ and $R_2$ are Gorenstein. Then $\Gamma$ is
Gorenstein if and only if $M_{12}$ is finitely generated and has
finite projective dimension on both sides.
\end{lem}

 Let $n\geq 2$, and let $\Gamma=\left(
          \begin{array}{cccc}
            R_1 & M_{12} & \cdots & M_{1n} \\
             & R_2 & \cdots & M_{2n} \\
             &  & \ddots & \vdots \\
                &  &  & R_n \\
          \end{array}
        \right)$ be an upper triangular matrix ring of order $n$.

\begin{note}\label{N}
\emph{Set $M_{1:}=\left(\begin{array}{c} M_{12}, \cdots , M_{1n}
\end{array}
            \right)$. We denote by $\Gamma_t=\left(
          \begin{array}{cccc}
            R_1 & M_{12} & \cdots & M_{1t} \\
             & R_2 & \cdots & M_{2t} \\
             &  & \ddots & \vdots \\
                &  &  & R_t \\
          \end{array}
        \right)$ the ring given by the $t\times t$ leading principal
        submatrix of $\Gamma$ for $1\leq t\leq n-1$. We denote by $\Gamma'_{n-1}$ the ring given by the $(n-1)\times (n-1)$ principal submatrix of $\Gamma$ which
        leaves out the first row and the first column. Denote the natural left
        $\Gamma_t$-module $\lbr M_{1,t+1}\\
\vdots \\ M_{t,t+1}\\ \rbr$ by $M_t^*$.}
\end{note}

The following result extends Lemma ~\ref{XChen}.

\begin{prop}\label{XChen1}
Let $\Gamma$ be an upper triangular matrix ring of order $n$ as
above. Assume that all $R_i$ are Gorenstein. Then $\Gamma$ is
Gorenstein if and only if all bimodules $M_{ij}$ are finitely
generated and have finite projective dimension on both sides.
\end{prop}

\begin{proof}[Proof of the ``only if" part.] Assume that $\Gamma$ is Gorenstein. We use induction on $n$.
The case $n=2$ is due to Lemma ~\ref{XChen}. Assume that $n>2$.
Write $\Gamma=\left(
              \begin{array}{cc}
                \Gamma_{n-1}  & M^*_{n-1} \\
                         0     &    R_n \\
              \end{array}
            \right)=\left(
              \begin{array}{cc}
                R_1  & M_{1:} \\
                   0  &  \Gamma'_{n-1}  \\
              \end{array}
            \right)$. Then $\Gamma_{n-1}$ and $\Gamma'_{n-1}$ are
            Gorenstein by Lemma~\ref{LE}, and $M^*_{n-1}$ and $M_{1:}$ are finitely generated and have finite
projective dimension on both sides by Lemma~\ref{XChen}. By
induction, all $M_{ij}$ possibly except for $M_{1n}$ are finitely
generated and have finite projective dimension on both sides. Since
$M_{1n}$ is a direct summand of $M^*_{n-1}$ as a right $R_n$-module,
it is finitely generated as a right $R_n$-module, and $\pd
(M_{1n})_{R_n} < \infty$. Since $M_{1n}$ is a direct summand of
$M_{1:}$ as a left $R_1$-module, it is finitely generated as a left
$R_1$-module, and $\pd_{R_1} M_{1n} < \infty$.
\end{proof}

We prove the ``if" part of Proposition~\ref{XChen1} together with
the following lemma.
\begin{lem}\label{XChen2}
Let $\Gamma$ be an upper triangular matrix ring of order $n$. Assume
that all bimodules $M_{ij}$ are finitely generated and have finite
projective dimension on both sides.
\begin{enumerate}
\item Let $X=\lbr X_1\\ \vdots \\ X_n\\
\rbr$ be a left $\Gamma$-module. If each $X_i$ is finitely generated
as a left $R_i$-module satisfying $\pd{_{R_i}X_i}<\infty$, then $X$
is finitely generated as a left $\Gamma$-module satisfying
$\pd_\Gamma X<\infty$.

 \item Let $Y=\left(\begin{array}{c} Y_1, \cdots , Y_n
\end{array} \right)$ be a right $\Gamma$-module. If each $Y_i$ is finitely generated as a right
$R_i$-module satisfying $\pd(Y_i)_{R_i}<\infty$, then $Y$ is
finitely generated as a right $\Gamma$-module satisfying $\pd
Y_\Gamma<\infty$.

\end{enumerate}
\end{lem}

\begin{proof}[ Proof of the ``if" part of Proposition {\rm ~\ref{XChen1}} and
Lemma {\rm~\ref{XChen2}}.] We use induction on $n$.

If $n=2$, Lemma
~\ref{XChen} implies that $\Gamma$ is Gorenstein. Let $X=\lbr X_1 \\ X_2 \\
\rbr$ be a left $\Gamma$-module. The following exact sequence of
$\Gamma$-modules
\[0\rightarrow \lbr X_1 \\ 0\\ \rbr\rightarrow X \rightarrow \lbr 0\\ X_2\\ \rbr \rightarrow 0\]
implies that $X$ is finitely generated, since $\lbr X_1 \\ 0 \\
\rbr$ and $\lbr  0 \\ X_2 \\
\rbr$ are finitely generated. We have $\id_{R_2}X_2 <
\infty$, since $R_2$ is Gorenstein and $\pd_{R_2}X_2 < \infty$. By Lemma~\ref{ExtFun}(1), $\id_\Gamma \lbr 0 \\ X_2 \\
\rbr=\id_{R_2}X_2 < \infty$, hence  $\pd_\Gamma \lbr 0 \\ X_2 \\
\rbr < \infty$. By Lemma~\ref{ExtFun}(1), $\pd_\Gamma \lbr X_1 \\ 0 \\
\rbr=\pd_{R_1}X_1 < \infty$. Then by the above exact sequence, we
have $\pd_\Gamma X < \infty$. The proof of Lemma \ref{XChen2}(2) is
similar.

Assume that $n>2$. Write $\Gamma=\left(\begin{array}{cc}
\Gamma_{n-1} & M^*_{n-1}\\
  0 &   R_n\\
\end{array}\right)
=\left(\begin{array}{cc}
R_1 & M_{1:}\\
  0 &   \Gamma'_{n-1}\\
\end{array}\right).$ Let $X$ be a left $\Gamma$-module and $Y$ be a right
$\Gamma$-module. Write $X=\lbr X' \\ X_n\\ \rbr$ and $ Y=(Y_1,Y')$,
 where $X'=\lbr X_1\\ \vdots \\ X_{n-1}\\
\rbr$ and $Y'=\left(\begin{array}{c} Y_2, \cdots , Y_n \end{array}
\right)$. By induction, $\Gamma_{n-1}$ and $\Gamma'_{n-1}$ are
Gorenstein, $X'$ is finitely generated as a left
$\Gamma_{n-1}$-module satisfying $\pd_{\Gamma_{n-1}}X' < \infty$,
and $Y'$ is finitely generated as a right $\Gamma'_{n-1}$-module
satisfying $\pd (Y')_{\Gamma'_{n-1}} < \infty$. We consider the left
$\Gamma_{n-1}$-module $M^*_{n-1}$. By induction in Lemma
\ref{XChen2}(1), $M^*_{n-1}$ is finitely generated as a left
$\Gamma_{n-1}$-module satisfying $\pd_{\Gamma_{n-1}}M^*_{n-1} <
\infty$. We observe that $M^*_{n-1}$ is finitely generated as a
right $R_n$-module satisfying $\pd(M^*_{n-1})_{R_n} < \infty$. By
Lemma~\ref{XChen}, $\Gamma$ is Gorenstein. This proves the ``if"
part of Proposition ~\ref{XChen1} in the general case.

For Lemma~\ref{XChen2}(1), write $X=\lbr X' \\ X_n\\ \rbr$, where by
induction $X'$ is finitely generated as a left $\Gamma_{n-1}$-module
satisfying $\pd_{\Gamma_{n-1}}X' < \infty$, and $X_n$ is finitely
generated as a left $R_n$-module satisfying $\pd{_{R_n}X_n}<\infty$.
Then $X$ is finitely generated as a left $\Gamma$-module satisfying
$\pd_{\Gamma}X < \infty$. The proof of Lemma \ref{XChen2}(2) in this
general case is similar.
\end{proof}

We give a characterization of left $\Gamma$-modules with finite
projective dimension; compare \cite[Proposition 2.8(1)]{EC}.
\begin{cor}\label{ProjFiniteMatModOverGore}
Let $\Gamma$ be a Gorenstein upper triangular matrix ring of order
$n$ with each $R_i$ Gorenstein. Let $X=\lbr X_1\\ \vdots \\ X_n\\
\rbr$ be a finitely generated left $\Gamma$-module. Then $\pd_\Gamma
X<\infty$ if and only if $\pd_{R_i} X_i<\infty$ for each $1\leq
i\leq n$.
\end{cor}

\begin{proof}
The ``if" part is due to Lemma~\ref{XChen2}(1). For the ``only if"
part, we only prove the case $n=2$. The general case is proved by
induction. Let $X=\lbr X_1\\ X_2\\ \rbr$ be a left $\Gamma$-module.
By Lemma~\ref{ExtFun}(2), $\pd_{R_2} X_2\leq \pd_\Gamma X< \infty$
and $\id_{R_1} X_1\leq \id_\Gamma X<\infty$. Then $\pd_{R_1}
X_1<\infty$ since $R_1$ is Gorenstein.
\end{proof}

The following results estimate the selfinjective dimension of an
upper triangular matrix ring; compare ~\cite[Remark 3.5]{XChen}.

\begin{prop}\label{GD}
 Let $\Gamma=\left(\begin{array}{cc}
R_1 & M_{12}\\
  0 &   R_2\\
\end{array}\right)$ be an upper triangular matrix ring with $R_1$ and $R_2$
Gorenstein. Let $m, d_1, d_2 \geq 0$.
\begin{enumerate}
\item If $\Gamma$ is $m$-Gorenstein, then both $R_1$ and $R_2$ are
$m$-Gorenstein. Moreover, $\pd_{R_1}M_{12}\leq m-1$ if $m\geq 1$.

\item Assume that $R_i$ is $d_i$-Gorenstein for $i=1,2$. If $M_{12}$
is finitely generated and projective on both sides, then $\Gamma$ is
$d$-Gorenstein, where $d=\max\{d_1,d_2\}$ if $d_1\neq d_2$, and
$d=d_1+1$ if $d_1=d_2$.

\end{enumerate}
\end{prop}

\begin{proof}

(1) We have $\id_\Gamma \lbr R_1\\0\rbr\leq m$, since $\Gamma$ is
$m$-Gorenstein and $\lbr R_1\\0\rbr$ is a projective
$\Gamma$-module. By Lemma~\ref{ExtFun}(2), $\id{_{R_1}}R_1\leq
\id_\Gamma \lbr R_1\\0\rbr\leq m$. Then $R_1$ is $m$-Gorenstein.
Recall by Lemma~\ref{LE}, $\pd_{R_1}M_{12}< \infty$. By
Lemma~\ref{ExtFun}(1), $\pd_\Gamma\lbr
M_{12}\\0\rbr=\pd_{R_1}M_{12}< \infty$. Then the following exact
sequence of $\Gamma$-modules
\[0\rightarrow \lbr M_{12}\\0\rbr \rightarrow \lbr M_{12}\\R_2\rbr
\rightarrow \lbr 0\\ R_2\rbr \rightarrow 0\]
 implies $\pd_\Gamma\lbr 0\\
R_2\rbr < \infty$, since $\lbr M_{12}\\R_2\rbr$ is a projective $\Gamma$-module. Then we have $\id_\Gamma \lbr 0\\
R_2\rbr \leq m$ since $\Gamma$ is $m$-Gorenstein. By
Lemma~\ref{ExtFun}(1), $\id_{R_2}R_2=\id_\Gamma\lbr 0\\R_2\rbr \leq
m$. Then $R_2$ is $m$-Gorenstein. Moreover, by the above exact
sequence, $\pd_\Gamma\lbr M_{12}\\0\rbr =0$ or $\pd_\Gamma\lbr
M_{12}\\0\rbr=\pd_\Gamma\lbr 0\\R_2\rbr-1\leq m-1$ for $m\geq 1$.
Hence by Lemma~\ref{ExtFun}(1), $\pd_{R_1}M_{12}\leq m-1$ for $m\geq
1$.

(2) By Lemma~\ref{XChen}, $\Gamma$ is Gorenstein. By the following
exact sequence of left $\Gamma$-modules
\[0\rightarrow \lbr R_1\\ 0 \rbr \rightarrow j_1(R_1)=\lbr R_1\\ \Hom_{R_1}(M_{12},R_1)\rbr
\rightarrow \lbr 0\\\Hom_{R_1}(M_{12},R_1)\rbr \rightarrow 0,\] we
have \[\id_\Gamma \lbr R_1\\ 0 \rbr \leq \max\left\{\id_\Gamma
j_1(R_1),
 \id_\Gamma \lbr 0\\\Hom_{R_1}(M_{12},R_1)\rbr+1\right\}.\]
Observe that for any left $R_1$-module $X$,
$\id{_{R_2}}(\Hom_{R_1}(M_{12},X)) \leq \id{_{R_1}} X$ since
$_{R_1}(M_{12})_{R_2}$ is projective on both sides. Then we have
$\id{_{R_2}}(\Hom_{R_1}(M_{12},R_1)) \leq \id{_{R_1}} R_1 \leq d_1$.
We have $\id{_{R_2}}(\Hom_{R_1}(M_{12},R_1)) \leq d_2$, since $R_2$
is $d_2$-Gorenstein. By Lemma~\ref{ExtFun}(1), $\id_\Gamma \lbr
0\\\Hom_{R_1}(M_{12},R_1)\rbr=\id{_{R_2}}(\Hom_{R_1}(M_{12},R_1))
\leq \min\{d_1,d_2\}$. The exact functor $j_1$ preserves injective
modules. Then we have $\id_\Gamma j_1(R_1)  \leq \id{_{R_1}} R_1\leq
d_1$. Hence $\id_\Gamma \lbr R_1\\ 0\\ \rbr \leq
\max\{d_1,\min\{d_1,d_2\}+1\}\leq d$. We have $\id_\Gamma \lbr
M_{12}\\ 0\\ \rbr \leq d$, since $M_{12}$ is a finitely generated
projective left $R_1$-module.

 By Lemma~\ref{ExtFun}(1), $\id_\Gamma \lbr 0\\ R_2\\
\rbr=\id_{R_2}R_2 \leq d_2\leq d$ since $R_2$ is $d_2$-Gorenstein.
Then the following exact sequence of $\Gamma$-modules
\[0\rightarrow \lbr M_{12}\\0\rbr \rightarrow \lbr M_{12}\\R_2\rbr
\rightarrow \lbr 0\\ R_2\rbr \rightarrow 0\]
 implies $\id_\Gamma\lbr M_{12}\\
R_2\rbr \leq d$. Hence $\id_{\Gamma}\Gamma \leq d$ since
$\Gamma=\lbr R_1\\0\rbr\oplus \lbr M_{12}\\R_2\rbr$ as left
$\Gamma$-modules. Then $\Gamma$ is $d$-Gorenstein .
\end{proof}

The following results estimate the selfinjective dimension of an
upper triangular matrix ring in general case. Recall the
$\Gamma_t$-module $M_t^*$ from Notation~\ref{N}.
\begin{prop}\label{LG1}
Let $\Gamma$ be an upper triangular matrix ring of order $n$ with
each $R_i$ Gorenstein. Let $m\geq 0$ and $d_i\geq 0$ for each $i$.
\begin{enumerate}
\item If $\Gamma$ is $m$-Gorenstein, then each $ R_i$ is
$m$-Gorenstein. Moreover, $\pd_{\Gamma_t}M_t^*\leq m-1$ if $m\geq 1$
for $1\leq t\leq n-1$.

\item Assume that $R_i$ is $d_i$-Gorenstein for $1\leq i \leq n$. If
all bimodules $M_{ij}$ are finitely generated and projective on both
sides, and each $M_t^*$ is a projective left $\Gamma_{t}$-module for
$1\leq t\leq n-1$, then $\Gamma$ is $d$-Gorenstein, where
$d=\max\{d_1,d_2,\cdots,d_n\}+1$.
\end{enumerate}
\end{prop}
\begin{proof}
(1) We use induction on $n$. The case $n=2$ is due to
Proposition~\ref{GD}(1). Assume that $n> 2$. Write
$\Gamma=\left(\begin{array}{cc}
\Gamma_{n-1} & M^*_{n-1}\\
  0 &   R_n\\
\end{array}\right)$. By Lemma~\ref{LE}, $\Gamma_{n-1}$ is Gorenstein. Then by
Proposition~\ref{GD}(1), we infer that $\Gamma_{n-1}$ and $R_n$ are
$m$-Gorenstein, and $\pd_{\Gamma_{n-1}}M_{n-1}^* \leq m-1$ if $m\geq
1$. By induction, we are done.

(2) We use induction on $n$. The case $n=2$ is due to
Proposition~\ref{GD}(2). Assume that $n> 2$. Write
$\Gamma=\left(\begin{array}{cc}
\Gamma_{n-1} & M^*_{n-1}\\
  0 &   R_n\\
\end{array}\right)$. By induction, $\Gamma_{n-1}$ is $d'$-Gorenstein,
where $d'=\max\{d_1,d_2,\cdots,d_{n-1}\}+1$. Since $M_{n-1}^*$ is a
projective left $\Gamma_{n-1}$-module and a projective right
$R_n$-module, we have that $\Gamma$ is $d''$-Gorenstein by
Proposition~\ref{GD}(2), where $d''=\max\{d',d_n\}$ if $d'\neq d_n$,
and $d''=d'+1$ if $d'=d_n$. In particular, we observe that $d''\leq
\max\{d_1,d_2,\cdots,d_n\}+1=d$.
\end{proof}

\section{Free EI categories}

In this section, we give a new characterization of finite free EI
categories in terms of the corresponding triangular matrix algebras.

Let $k$ be a field. Let $\C$ be a finite category, that is, it has
only finitely many morphisms, and consequently it has only finitely
many objects. Denote by $\Mor\C$ the finite set of all morphisms in
$\C$. The \emph{category algebra} \emph{k}$\C$ of $\C$ is defined as
follows: $\emph{k}\C=\bigoplus\limits_{\alpha \in \Mor\C}k\alpha$ as
a $k$-vector space and the product $*$ is given by the rule
 \[\alpha * \beta=\left\{\begin{array}{ll}
                     \alpha\circ\beta, & \text{ if }\text{$\alpha$ and $\beta$ can be composed in $\C$}; \\
                     0, & \text{otherwise.}
                   \end{array}\right.\]
The unit is given by $1_{k\C}=\sum\limits_{x \in \Obj\C }\Id_x$,
where $\Id_x$ is the identity endomorphism of an object $x$ in $\C$.

Let $\C$ and $\mathscr D$ be finite categories. If they are
equivalent, then \emph{k}$\C$ and \emph{k}$\mathscr D$ are Morita
equivalent; see \cite[Proposition 2.2]{PWebb2}. In particular, $k\C$
is Morita equivalent to $k\C_0$, where $\C_0$ is any skeleton of
$\C$. So we may assume that $\C$ is \emph{skeletal}, that is, for
any two distinct objects $x$ and $y$ in $\C$, $x$ is not isomorphic
to $y$.

The category $\C$ is called a \emph{finite EI category} provided
that all endomorphisms in $\C$ are isomorphisms. In particular,
$\Homc(x,x)=\Ac(x)$ is a finite group for any object $x$ in $\C$.

In what follows, we assume that $\C$ is a finite EI category which
is skeletal.

Let $\C$ have $n$ objects with  $n\geq 2$. We assume that
$\Obj\C=\{x_1,x_2,\cdots,x_n\}$ satisfying
$\Homc(x_i,x_j)=\emptyset$ if $i<j$. Let $M_{ij}=k\Homc(x_j,x_i)$.
Write $R_i=M_{ii}$. We observe that
$R_i=(\Id_{x_i})k\C(\Id_{x_i})=k\Ac(x_i)$ is a group algebra.
Then $M_{ij}$ is naturally an $R_i$-$R_j$-bimodule, and we have a
morphism of $R_i$-$R_j$-bimodules $\psi_{ilj}: M_{il}\otimes_{R_l}
M_{lj} \rightarrow M_{ij}$ which is induced by the composition of
morphisms in $\C$.
\begin{note}\label{NOTE}
\emph{The category algebra $k\C$ is isomorphic to the corresponding
upper triangular matrix algebra $\Gamma_{\C}=\left(
          \begin{array}{cccc}
            R_1 & M_{12} & \cdots & M_{1n} \\
             & R_2 & \cdots & M_{2n} \\
             &  & \ddots & \vdots \\
                &  &  & R_n \\
          \end{array}
        \right)$.
Let $\Gamma_t$ be the algebra given by the $t\times t$ leading
principal
 submatrix of $\Gamma_{\C}$. Denote the left
 $\Gamma_t$-module $\lbr M_{1,t+1}\\
\vdots \\ M_{t,t+1}\\ \rbr$ by $M_t^*$, for $1\leq t\leq n-1$.} 
\end{note}

\begin{defi}\label{DefOfProjCat}
\emph{Let $\C$ be a finite EI category with
$\Obj\C=\{x_1,x_2,\cdots,x_n\}$ satisfying
$\Homc(x_i,x_j)=\emptyset$ if $i<j$. We say that $\C$ is} projective
over $k$ \emph{if each $M_{ij}=k\Homc(x_j,x_i)$ is a projective left
$R_i$-module and a projective right $R_j$-module for $1\leq i<j\leq
n$.}
\end{defi}

Let $G$ be a finite group. We assume that $G$ has a left action on a
finite set $X$. For any $x\in X$, denote its stabilizer by ${\rm
Stab}(x)=\{g\in G \mid g.x=x\}$. The vector space $kX$ is a natural
$kG$-module.

The following result is well known, which can be deduced from
\cite[II.5 Theorem 6]{JA}. We give an elementary argument for
completeness; compare the third paragraph of the proof of
\cite[Theorem 2.5]{PWebb1}.

\begin{lem}\label{GT}
The $kG$-module $kX$ is projective if and only if the order of each
stabilizer ${\rm Stab}(x)$ is invertible in $k$.
\end{lem}
\begin{proof}
We may assume that the action on $X$ is transitive. Take $x\in X$,
we have $X\simeq G/{\rm Stab}(x)$. Then we have an isomorphism
$kX\simeq k(G/{\rm Stab}(x))$ of $kG$-modules. We observe the
following Maschke-type result: for any subgroup $H$ of $G$, the
canonical projection $kG\twoheadrightarrow k(G/H)$ of $kG$-modules
splits if and only if the order of $H$ is invertible in $k$. Then
the lemma follows immediately.
\end{proof}

Let $\C$ be a finite EI category and $\alpha \in \Homc(x,y)$. We
call $L_{\alpha}=\{g\in \Ac(y) \mid g\circ\alpha=\alpha\}$ the
\emph{left stabilizer} of $\alpha$, and $R_{\alpha}=\{h\in \Ac(x)
\mid \alpha\circ h=\alpha\}$ the \emph{right stabilizer} of
$\alpha$. Then we have the following immediate consequence of Lemma
~\ref{GT}.

\begin{cor}\label{PI}
Let $\C$ be a finite EI category. Then $\C$ is projective if and
only if for any $\alpha \in \Mor\C$, the orders of $L_{\alpha}$ and
$R_{\alpha}$ are invertible in $k$. \hfill $\square$
\end{cor}

Let $\C$ be a finite EI category. Recall from ~\cite[Definition
2.3]{LLi} that a morphism $x\overset{\alpha}{\rightarrow} y$ in $\C$
is \emph{unfactorizable} if $\alpha$ is not an isomorphism and
whenever it has a factorization as a composite $x
\overset{\beta}{\rightarrow} z \overset{\gamma}{\rightarrow} y$,
then either $\beta$ or $\gamma$ is an isomorphism. Let
$x\overset{\alpha}{\rightarrow} y$ in $\C$ be an unfactorizable
morphism. Then $h\circ\alpha\circ g$ is also unfactorizable for
every $h \in \Ac(y)$ and every $g \in \Ac(x)$; see \cite[Proposition
2.5]{LLi}. Let $x\overset{\alpha}{\rightarrow} y$ in $\C$ be a
morphism with $x\neq y$. Then it has a decomposition
$x=x_0\overset{\alpha_1}{\rightarrow}
x_1\overset{\alpha_2}{\rightarrow} \cdots
\overset{\alpha_n}{\rightarrow} x_n=y$ with all $\alpha_i$
unfactorizable; see \cite[Proposition 2.6]{LLi}.

Following \cite[Definition 2.7]{LLi}, we say that a finite EI
category $\C$ satisfies the Unique Factorization Property (UFP), if
whenever a non-isomorphism $\alpha$ has two  decompositions into
unfactorizable morphisms:
  \[ x=x_0\overset{\alpha_1}{\rightarrow}
x_1\overset{\alpha_2}{\rightarrow} \cdots
\overset{\alpha_m}{\rightarrow} x_m=y\] and
\[x=y_0\overset{\beta_1}{\rightarrow}
y_1\overset{\beta_2}{\rightarrow} \cdots
\overset{\beta_n}{\rightarrow} y_n=y,\] then $m=n$, $x_i=y_i$, and
there are $h_i\in \Ac(x_i)$, $1\leq i\leq n-1$ such that
$\beta_1=h_1\circ\alpha_1$, $\beta_2=h_2\circ\alpha_2\circ
h_1^{-1}$, $\cdots$, $\beta_{n-1}=h_{n-1}\circ\alpha_{n-1}\circ
h_{n-2}^{-1}$, $\beta_n=\alpha_n\circ h_{n-1}^{-1}$.

Let $\C$ be a finite EI category. Following \cite{LLi}, we say that
$\C$ is a finite \emph{free} EI category if it satisfies the UFP.
This is an equivalent characterization of finite free EI categories
in ~\cite[Proposition 2.8]{LLi}.

Let $\Gamma=\Gamma_{\C}$ be the corresponding upper triangular
matrix algebra of $\C$. We recall the $\Gamma_t$-module $M_t^*$ from
Notation ~\ref{NOTE} for each $1\leq t\leq n-1$. We have the
following characterization of finite free EI categories.

\begin{prop}\label{MatInterpOfFreeness}
Let $\C$ be a finite skeletal EI category with
$\Obj\C=\{x_1,x_2,\cdots,x_n\}$ satisfying
$\Homc(x_i,x_j)=\emptyset$ if $i<j$. Assume that $\C$ is projective.
Then $\C$ is free if and only if each $M_t^*$ is a projective left
$\Gamma_t$-module for $1\leq t \leq n-1$.
\end{prop}

Before giving the proof of the proposition, we make some
preparations.

\begin{defi}
\emph{Let $\C$ be a finite EI category and $x\in {\rm Obj}\C$. We
say that $\C$ is
 }free from $x$ \emph{if whenever an arbitrary non-isomorphism $x\overset{\alpha}{\rightarrow}
y$ in $\C$ has two decompositions $ x\overset{\alpha_1}{\rightarrow}
z_1\overset{\alpha_2}{\rightarrow}y$ and
$x\overset{\beta_1}{\rightarrow} z_2\overset{\beta_2}{\rightarrow}
y$ with $\alpha_1$ and $\beta_1$ unfactorizable, then $z_1=z_2$ and
there is an endomorphism $h\in \Ac(z_1)$ such that $\beta_1=h\circ
\alpha_1$ and $\beta_2=\alpha_2 \circ h^{-1}$.}
\end{defi}

\begin{lem}\label{EC2}
Let $\C$ be a finite EI category. Then $\C$ is free if and only if
$\C$ is free from any object.
\end{lem}

\begin{proof}
The ``only if'' part is trivial. For the ``if'' part, assume that
$\C$ is free from any object. Let $x\overset{\alpha}{\rightarrow} y$
be a non-isomorphism in $\C$. Assume that $\alpha$ has two
decompositions into unfactorizable morphisms:
  \[ x=x_0\overset{\alpha_1}{\rightarrow}
x_1\overset{\alpha_2}{\rightarrow} \cdots
\overset{\alpha_m}{\rightarrow} x_m=y\] and
\[x=y_0\overset{\beta_1}{\rightarrow}
y_1\overset{\beta_2}{\rightarrow} \cdots
\overset{\beta_n}{\rightarrow} y_n=y.\] Since $\C$ is free from $x$,
we have $x_1=y_1$, and there is an endomorphism $h_1\in \Ac(x_1)$
such that $\beta_1=h_1\circ\alpha_1$ and $\alpha_m\circ \cdots
\alpha_2=\beta_n\circ \cdots \beta_2\circ h_1$. We continue this
process. We obtain that $m=n$, $x_i=y_i$, and there are $h_i\in
\Ac(x_i)$, $1\leq i\leq m-1$ such that $\beta_1=h_1\circ\alpha_1$,
$\beta_2=h_2\circ\alpha_2\circ h_1^{-1}$, $\cdots$,
$\beta_{m-1}=h_{m-1}\circ \alpha_{m-1}\circ h_{m-2}^{-1}$,
$\beta_m=\alpha_m\circ h_{m-1}^{-1}$. Then $\C$ is free.
\end{proof}

Let $W_{il}\subseteq \Homc(x_l,x_i)$ and $T_{lj}\subseteq
\Homc(x_j,x_l)$ be subsets. Denote the subset $W_{il}\circ T_{lj} =
\{f\circ g\mid f\in W_{il} \text{ and } g\in T_{lj}\}\subseteq
\Homc(x_j,x_i)$.

\begin{note}\label{NO}
\emph{Set $\Ho(x_j,x_i)=\{\alpha\in \Homc(x_j,x_i)\mid \alpha \text{
is unfactorizable}\}$. Denote $M^0_{ij}=k\Ho(x_j,x_i)$. Then
$M^0_{ij}$ is an $R_i$-$R_j$-subbimodule of $M_{ij}$. Moreover,
\[M_{ij}=M^0_{ij}\oplus (\sum\limits_{l=i+1}^{j-1}
  k(\Homc(x_l,x_i)\circ\Ho(x_j,x_l)))\] as an $R_i$-$R_j$-bimodule.}
\end{note}

\begin{lem}\label{DU}
Let $\C$ be a finite EI category with
$\Obj\C=\{x_1,x_2,\cdots,x_n\}$ satisfying
$\Homc(x_i,x_j)=\emptyset$ if $i<j$. Assume that $\C$ is free from
$x_j$. Then for any $1\leq i< j$, we have
\[\Homc(x_j,x_i)=\bigsqcup\limits_{l=i}^{j-1}
  (\Homc(x_l,x_i)\circ\Ho(x_j,x_l)),\] where the right hand side is
  a disjoint union.
\end{lem}

\begin{proof}
Recall that the category $\C$ is free from $x_j$. Then we have
\[\Homc(x_{l_1},x_i)\circ\Ho(x_j,x_{l_1})\cap
\Homc(x_{l_2},x_i)\circ\Ho(x_j,x_{l_2})= \phi\] for $l_1\neq l_2$.
Since every morphism can be decomposed as a composition of
unfactorizable morphisms, we have the required equation.
\end{proof}

We observe that there is a surjective morphism
\begin{align}\label{SUR1}
\xi_{ilj}:M_{il}\otimes_{R_l} M^0_{lj}\longrightarrow
k(\Homc(x_l,x_i)\circ\Ho(x_j,x_l))
\end{align}
of $R_i$-$R_j$-bimodules sending $\beta\otimes \alpha$ to
$\beta\circ \alpha$ for $i<l<j$.

We recall the $\Gamma_{n-1}$-module $i_t(A)$ for each $1\leq t\leq
n-1$ and each $R_t$-module $A$; see ~(\ref{ISO}). Then we have the
following natural surjective morphisms
\begin{align}\label{SUR2}
\Phi_i:M^0_{in}\oplus(\bigoplus\limits_{l=i+1}^{n-1}
(M_{il}\otimes_{R_l} M^0_{ln}))\longrightarrow M_{in}
\end{align}
of $R_i$-$R_n$-bimodules induced by $\xi_{iln}$ for $1\leq i<l\leq
n-1$, and
\begin{align}\label{SUR3}
\Phi:\bigoplus\limits_{t=1}^{n-1}i_t(M^0_{tn})\longrightarrow
M^*_{n-1}
\end{align}
of $\Gamma_{n-1}$-$R_n$-bimodules induced by $\Phi_i$ for $1\leq
i\leq n-1$.

\begin{lem}\label{PC}
Let $\C$ be a finite EI category with
$\Obj\C=\{x_1,x_2,\cdots,x_n\}$ satisfying
$\Homc(x_i,x_j)=\emptyset$ if $i<j$. Assume that $\C$ is projective.
Then the above surjective morphism $\Phi$ is a projective cover of
the left $\Gamma_{n-1}$-module $M^*_{n-1}$.
\end{lem}
\begin{proof}
Since the category $\C$ is projective, we have that $M_{tn}$ is a
projective left $R_t$-module for each $1\leq t\leq n-1$. Then each
$M^0_{tn}$ is a projective left $R_t$-module, since it is a direct
summand of $M_{tn}$; see Notation \ref{NO}. By
Proposition~\ref{ProjInjOverTria}(1), $\bigoplus \limits_{t=1}^{n-1}
i_t(M^0_{tn})$ is a projective left $\Gamma_{n-1}$-module.

To prove that $\Phi$ is a projective cover, it suffices to show that
${\rm top}(\bigoplus\limits_{t=1}^{n-1}i_t(M^0_{tn}))$ and ${\rm
top}(M^*_{n-1})$ are isomorphic. Here, we write ${\rm top}X=X/{\rm
rad}X$ for a module $X$, where ${\rm rad}X$ denotes the radical of
$X$. Recall ${\rm rad}(\Gamma_{n-1})=\left(
          \begin{array}{cccc}
            {\rm rad}(R_1) & M_{12} & \cdots & M_{1,n-1} \\
             & {\rm rad}(R_2) & \cdots & M_{2,n-1} \\
             &  & \ddots & \vdots \\
                &  &  & {\rm rad}(R_{n-1}) \\
          \end{array}
        \right)$. By ${\rm
rad}(\bigoplus\limits_{t=1}^{n-1}i_t(M^0_{tn}))={\rm
rad}(\Gamma_{n-1})(\bigoplus\limits_{t=1}^{n-1}i_t(M^0_{tn}))$, we
compute that the $i$-th component of ${\rm
top}(\bigoplus\limits_{t=1}^{n-1}i_t(M^0_{tn}))$ is isomorphic to
$M^0_{in}/{\rm rad}(R_i)M^0_{in}$. By a similar calculation, we have
that the $i$-th component of ${\rm top}(M^*_{n-1})$ is isomorphic to
$M^0_{in}/{\rm rad}(R_i)M^0_{in}$. Then we have the required
isomorphism.
\end{proof}

\begin{lem}\label{EC}
Let $\C$ be a finite EI category with
$\Obj\C=\{x_1,x_2,\cdots,x_n\}$ satisfying
$\Homc(x_i,x_j)=\emptyset$ if $i<j$. Assume that $\C$ is projective.
Then the following are equivalent.
\begin{enumerate}
\item  The category $\C$ is free from $x_n$.
\item  All the surjective morphisms $\Phi_i$ are isomorphisms.
\item  The surjective morphism $\Phi$ is an isomorphism.
\item  The left $\Gamma_{n-1}$-module $M^*_{n-1}$ is projective.
\end{enumerate}
\end{lem}

\begin{proof}
``(1)$\Rightarrow$ (2)" Since the category $\C$ is free from $x_n$,
by Notation \ref{NO} and Lemma \ref{DU} we have
\[M_{in}=M^0_{in}\oplus (\bigoplus\limits_{l=i+1}^{n-1}
  k(\Homc(x_l,x_i)\circ\Ho(x_n,x_l))).\] Since
  $\Phi_i$ is induced by $\xi_{iln}$, we only need to prove that
  $\xi_{iln}$ is an isomorphism for each $i<l<n$. Indeed, since $\C$ is free from $x_n$, we
have $\beta'\circ \alpha'=\beta\circ \alpha$ in
$\Homc(x_l,x_i)\circ\Ho(x_n,x_l)$ if and only if $\beta'=\beta\circ
g$ and $\alpha'=g^{-1}\circ \alpha$ for some $g\in \Ac(x_l)$. Then
we have a well-defined morphism
\[\eta_{iln}: k(\Homc(x_l,x_i)\circ\Ho(x_n,x_l)) \longrightarrow
M_{il}\otimes_{R_l} M^0_{ln}\] of $R_i$-$R_n$-bimodules sending
$\beta\circ \alpha$ to $\beta\otimes \alpha$. It is directly verify
that $\xi_{iln}$ and $\eta_{iln}$ are mutually inverse. Then we have
the required isomorphisms.

``(2)$\Rightarrow$ (1)" Since each $\Phi_i$ is an isomorphism for
$1\leq i\leq n-1$, we have
\[\Homc(x_{l_1},x_i)\circ\Ho(x_n,x_{l_1})\cap
\Homc(x_{l_2},x_i)\circ\Ho(x_n,x_{l_2})= \phi\] for $i<l_1\neq l_2<
n$, which implies that $\C$ is free from $x_n$.

``(2)$\Leftrightarrow$ (3)" It is obvious, since $\Phi$ is induced
by $\Phi_1, \cdots, \Phi_{n-1}$.

``(3)$\Leftrightarrow$ (4)" Apply Lemma \ref{PC}.
\end{proof}

\begin{proof}[Proof of Proposition {\rm \ref{MatInterpOfFreeness}}.]
Assume that $\C$ is projective. For each $1\leq t \leq n-1$, we
consider the full subcategory $\C_{t}$ of $\C$ with
$\Obj\C_{t}=\{x_1,\cdots,x_t\}$. We observe that $\C_{t}$ is free
from $x_t$ if and only if $\C$ is free from $x_t$. Then by Lemma
\ref{EC}, we have that $\C$ is free from $x_t$ if and only if
$M_t^*$ is a projective left $\Gamma_t$-module. By Lemma \ref{EC2},
we are done.
\end{proof}

\section{The main results}

In this section, we give a necessary and sufficient condition on
when the category algebra $k\C$ of a finite EI category $\C$ is
Gorenstein, and when $k\C$ is $1$-Gorenstein.

Throughout this section, when the category $\C$ is skeletal, we
assume that $\Obj\C=\{x_1,x_2,\cdots,x_n\}$ satisfying
$\Homc(x_i,x_j)=\emptyset$ if $i<j$.
\begin{prop}\label{GP}
 Let $k$ be a field and $\C$ be a finite EI category. Then the category algebra $k\C$ is Gorenstein if and only if  $\C$ is projective over $k$.
\end{prop}
\begin{proof}
Without loss of generality, we assume that $\C$ is skeletal.
Otherwise, we take its skeleton $\C_0$, which is equivalent to $\C$.
We observe that  $\C$ is projective if and only if $\C_0$ is
projective and that $\C$ is Gorenstein if and only if $\C_0$ is
Gorenstein.

Let $\Gamma=\Gamma_{\C}$ be the corresponding upper triangular
matrix algebra of $\C$. Observe that $R_i=k\Ac(x_i)$ is a group
algebra of a finite group. In particular, it is a selfinjective
algebra. Then each $R_i$-$R_j$-bimodule $M_{ij}$ has finite
projective dimension on both sides if and only if it is projective
on both sides. Consequently, the statement is immediately due to
Proposition~\ref{XChen1}.
\end{proof}

\begin{exa}\label{EX1}
{\rm Let $G$ be a finite group and $\mathcal{P}$ a finite poset. We
assume that $\mathcal{P}$ is a \emph{$G$-poset}, that is, $G$ acts
on $\mathcal{P}$ by poset automorphisms. We recall that the
\emph{transporter category} $G\propto \mathcal{P}$ is defined as
follows. It has the same objects as $\mathcal{P}$, that is, ${\rm
Obj}(G\propto \mathcal{P})={\rm Obj}\mathcal{P}$. For $x, y \in {\rm
Obj}(G\propto \mathcal{P})$, a morphism from $x$ to $y$ is an
element $g$ in $G$ satisfying $gx\leq y$. The corresponding morphism
is denoted by $(g; gx\leq y)$. The composition of morphisms is given
by the multiplication in $G$.

We observe that $G\propto \mathcal{P}$ is a finite EI category.
Here, we use the fact that $gx\leq x$ implies $gx=x$. One can check
directly that if $\Hom_{G\propto \mathcal{P}}(x,y)\neq \phi$, then
both ${\rm Aut}_{G\propto \mathcal{P}}(x)$ and ${\rm Aut}_{G\propto
\mathcal{P}}(y)$ act freely on $\Hom_{G\propto \mathcal{P}}(x,y)$,
in particular, both the left and right stabilizers $L_{\alpha}$ and
$R_{\alpha}$ of a morphism $\alpha$ are trivial; compare
\cite[Definition 2.1]{XF1}. By {\rm Corollary \ref{PI}}, $G\propto
\mathcal{P}$ is projective over $k$. Then the category algebra
$k(G\propto \mathcal{P})$ is Gorenstein by {\rm Proposition
\ref{GP}}; compare\cite[Lemma 2.3.2]{XF2}.}
\end{exa}

\begin{thm}\label{1GP}
 Let $k$ be a field and $\C$ be a finite EI category. Then the category algebra $k\C$ is $1$-Gorenstein if and only if  $\C$ is free and projective over $k$.
\end{thm}

\begin{proof}
We assume that $\C$ is skeletal. The reason is similar to the first
paragraph in the proof of Theorem \ref{GP}.

 Let $\Gamma=\Gamma_{\C}$ be the corresponding upper
triangular matrix algebra of $\C$. As mentioned above, each
$R_i=k\Ac(x_i)$ is a selfinjective algebra. In particular, a module
over $R_i$ having finite projective dimension is necessarily
projective.

For the ``if'' part, assume that $\C$ is free and projective over
$k$. Then all bimodules $M_{ij}=k\Homc(x_j,x_i)$ are finitely
generated and projective on both sides. By
Proposition~\ref{MatInterpOfFreeness}, each $M_t^*$ is a projective
left $\Gamma_t$-module for $1\leq t\leq n-1$. Then $\Gamma$ is
$1$-Gorenstein by Proposition~\ref{LG1}(2).

For the ``only if'' part, assume that $\Gamma$ is $1$-Gorenstein.
Then Proposition~\ref{LG1}(1) implies that each $\Gamma_t$ is
$1$-Gorenstein and each $M_t^*$ is a projective left
$\Gamma_t$-module for $1\leq t\leq n-1$. By
  Corollary \ref{ProjFiniteMatModOverGore}, $\pd_{R_i} M_{i,t+1}< \infty$ for $1\leq i<t+1\leq n$, and thus $M_{i,t+1}$ is a projective left $R_i$-module.
   By Lemma~\ref{LE}, $\pd (M_t^*)_{R_{t+1}}< \infty$ for $1\leq t\leq
   n-1$, since $M_t^*$ is a projective left $\Gamma_t$-module. Since each $M_{ij}$ is a direct summand of $M_{j-1}^*$ as
   $R_j$-modules, we have $\pd (M_{ij})_{R_j}< \infty$ for $1\leq i<j\leq
   n$, and thus $M_{ij}$ is a projective right $R_j$-module. Then the category $\C$ is projective. Since each
$M_t^*$ is a projective left $\Gamma_t$-module for $1\leq t\leq
n-1$, the category $\C$ is free by
Proposition~\ref{MatInterpOfFreeness}.
\end{proof}

\begin{exa}
{\rm Let $\mathcal{P}$ be a finite poset. For two elements $x$ and
$y$, we denote $x<y$ if $x\leq y$ and $x\neq y$. By a} chain, {\rm
we mean a totally ordered set. We observe that $\mathcal{P}$ is free
as a category if and only if for any $x\leq y$ in $\mathcal{P}$, the
closed interval $[x,y]$ is a chain.}

{\rm Let $G$ be a finite group and $\mathcal{P}$ a finite $G$-poset.
Consider the transporter category $G\propto \mathcal{P}$. Recall
that a morphism $(g;gx\leq y)$ in $G\propto \mathcal{P}$ is an
isomorphism if and only if $gx=y$. We observe that a non-isomorphism
$(g;gx<y)$ in $G\propto \mathcal{P}$ is unfactorizable if and only
if there is no object $z \in {\rm Obj}\mathcal{P}$ such that
$gx<z<y$. We infer by the UFP that the transporter category
$G\propto \mathcal{P}$ is free if and only if the category
$\mathcal{P}$ is free. By Example \ref{EX1}, we have that the
transporter category $G\propto \mathcal{P}$ is projective. Then by
Theorem \ref{1GP}, the category algebra $k(G\propto \mathcal{P})$ is
$1$-Gorenstein if and only if the poset $\mathcal{P}$ is free as a
category. We mention that this result can be obtained by combining
\cite[Lemma 2.3.2]{XF2} and \cite[Proposition 2.2]{AR}.}
\end{exa}

\begin{exa}{\rm(\cite[ Thorem 5.3]{LLi})}\label{EX}
\emph{Let $\C$ be a finite EI category. Then the category algebra
$k\C$ is hereditary if and only if $\C$ is free satisfying that the
endomorphism groups of all objects have orders invertible in $k$}.

{\rm We may assume that $\C$ is skeletal. Let $\Gamma=\Gamma_{\C}$
be the corresponding upper triangular matrix algebra of $\C$. We
first claim that $\Gamma$ has finite global dimension if and only if
each $\Ac(x_i)$ has order invertible in $k$. In this case, the
category $\C$ is projective over $k$ by Corollary \ref{PI}. Indeed,
by \cite[Corollary 4.21 (4)]{DE}, $\Gamma$ has finite global
dimension if and only if each $R_i=k\Ac(x_i)$ has finite global
dimension, which is equivalent to that each $R_i=k\Ac(x_i)$ is
semi-simple. Then we have the claim.

We recall the well-known fact that a finite dimensional algebra is
hereditary if and only if it is $1$-Gorenstein with finite global
dimension. Then the required result follows from the above claim and
Theorem \ref{1GP}.}
\end{exa}

\section*{Acknowledgements}
The author is grateful to her supervisor Professor Xiao-Wu Chen for
his guidance. This work is supported by the National Natural Science
Foundation of China (No. 11201446), the Fundamental Research Funds
for the Central Universities (WK0010000039) and the Key Program of
Excellent Youth Foundation of Colleges in Anhui Province
(2013SQRL071ZD).


\end{document}